\documentclass[11pt]{article}
\usepackage{caption}
\usepackage{blindtext}
\usepackage{algorithm}
\usepackage{algorithmic}

\usepackage{xcolor}
\usepackage{soul}
\usepackage{lineno}
\usepackage{graphicx}
\usepackage{amsmath,amsthm,amsfonts,amssymb,amscd,cite}
\allowdisplaybreaks

\newtheorem{theorem}{Theorem}
\newtheorem{theo}{Theorem}

\newtheorem{remark}[theo]{Remark}

\date{}
\begin{document}

\title{\bf  A Novel Method to Construct \\
NSSD Molecular Graphs}

\author{{Umar Hayat}$^1$, {Mubasher Umer}$^1$, {
Ivan Gutman}$^2$, {Bijan Davvaz}$^{3}$, {\'Alvaro Nolla de Celis}$^{4}$\\
$^1$Department of Mathematics, Quaid-i-Azam University, \\
Islamabad, Pakistan \\
umar.hayat@qau.edu.pk \\
mumer@math.qau.edu.pk \\
$^2$Faculty of Science, University of Kragujevac, \\
P. O. Box 60, 34000 Kragujevac, Serbia\\
gutman@kg.ac.rs \\
$^3$Department of Mathematics, Yazd University, Yazd, Iran \\
 davvaz@yazd.ac.ir\\
$^4$Faculty of Teacher Training, Universidad Autónoma of Madrid, Spain \\
alvaro.nolla@uam.es }
\date{}
\maketitle

\begin{abstract}
A graph is said to be NSSD (=non-singular with a singular deck) if
it has no eigenvalue equal to zero, whereas all its
vertex--deleted subgraphs have eigenvalues equal to zero. NSSD
graphs are of importance in the theory of conductance of organic
compounds. In this paper, a novel method is described for
constructing NSSD molecular graphs from the commuting graphs of
the $H_v$-group. An algorithm is presented to construct the NSSD
graphs from these commuting graphs.\\ \\
{\bf Keywords:} H$_v$-group; Commuting graph; Non-singular graph
with singular deck. \\
AMS Mathematics Classification: 20N20.
\end{abstract}

\section{Introduction}

Beginning in 1970s, graph spectra found noteworthy applications in
chemistry, mainly in the area of molecular orbital theory
\cite{17,18}. One of the most recent developments along these
lines are the model of Fowler et al. \cite{new1}, describing the
electrical current created by the injection of ballistic electrons
via external contacts into an unsaturated conjugated molecule.
Within this model, the considered molecule is predicted to be an
insulator for all single-$\pi$-electron connections, if the
underlying molecular graph belongs to the class of NSSD graphs.
Let $G$ be a simple graph with vertex set
$V(G)=\{v_1,v_2,\ldots,v_n\}$ and edge set $E(G)$. Its adjacency
matrix $\mathbf A=(a_{ij})$ is defined so that $a_{ij}=1$ if the
vertices $v_i$ and $v_j$ are adjacent, and $a_{ij}=0$ otherwise
\cite{new2}. The eigenvalues of $\mathbf A$, denoted by
$\lambda_1,\lambda_2,\ldots,\lambda_n$ are said to be the
eigenvalues of the graph $G$ and to form the spectrum of $G$ \
\cite{new2}. The nullity of a graph $G$, denoted by $\eta(G)$, is
the number of eigenvalue that are equal to zero. If none of these
eigenvalues is equal to zero, $i.e.$, $\eta(G) = 0$ then the graph
is said to be non-singular. Otherwise, it is singular. The graph
$G$ is an NSSD graph (a {\bf N}on-{\bf S}ingular graph with a {\bf
S}ingular {\bf D}eck) if it is non-singular, and if all its
vertex-deleted subgraphs $G-v_i \ , \ i=1,2,\ldots,n$ are singular
\cite{14,15,new4}. The term NSSD was  introduced in \cite{new5},
motivated by the search for carbon molecules in the Huckel model.
The first step in the history of the development of hyperstructure theory
was the 8th congress of Scandinavian mathematician from $1934$, when Marty
\cite{10} put forward the concept of hypergroup, analyzed its properties
and showed its utility in the study of groups, algebraic functions,
and rational fractions. Eventually, hyperstructure theory found
applications in the field of cryptography, geometry, graphs,
hypergraphs, binary relations, theory of fuzzy, coding theory,
automata theory, etc. The correspondence between hyperstructure
and binary relations is implicity contained in Nieminen \cite{9}
who associated hypergroups to connected simple graphs; for further
work in this direction see \cite{6,8,12,13}. In 1990, Vougiouklis
introduced the concept of H$_v$-structure \cite{13}. The main idea
of $H_v$-structures is in establishing a generalization of the
other algebraic hyperstructures. In fact, some axioms related to
these hyperstructures are replaced by their corresponding weak
axioms.

Various classes of NSSD graphs and their construction are described in recent articles \cite{14,new6}. Some necessary and
sufficient conditions are obtained for a two-vertex-deleted subgraph of an NSSD
graph $G$ to remain an NSSD by considering triangles in the inverse NSSD $G^{-1}$
\cite{new7}. In this paper, we present a new method for constructing NSSD graphs,
utilizing hyperstructure theory and commuting graphs. We also present an
algorithm written in $GAP$ language to construct NSSD graphs from
these commuting graphs. The paper is structured as follows. In
Section $2$, we consider an $H_{v}$-group. We discuss its
commuting graphs and establish some NSSD graphs. We present some
algorithms. Using these algorithms we determine NSSD graphs. In
Section $3$, we find some NSSD molecular graphs from these
commuting graphs. Conclusions are made in Section $4$.

\section{Commuting Graphs on $H_v$-Group and an  Algorithm to Determine NSSD Graphs}

In this section we discuss some metric properties of commuting
graphs on $H_v$-group. Recall that in a commuting graph, two
elements are joined by an edge if they commute with each other.
For further study of commuting graphs see \cite{1,2,3,4,5,11}.

Let $J$ be a non-empty set. A hyperoperation on a non-empty set $J$ is a
mapping $\circ :J\times J\rightarrow \mathcal{P}^{\ast }(J)$, where $\mathcal{P}^{\ast }(J)$
denotes the set of all non-empty subsets of $J.$ If $U$, $V$ are non-empty subsets
of $J$ and $x\in J$, then we define
$$
U\circ V=\displaystyle \bigcup_{x\in U \atop y\in V}x\circ y, \
x\circ V=\left\{ x\right\}\circ V \ {\rm and} \  V\circ x=V\circ
\left\{ x\right\}.
$$

An algebraic hyperstructure $(J,\circ)$ is said to be an
$H_v$-group if it satisfies the following properties
\begin{itemize}
\item[(1)] $ (J,\circ)$ is weakly associative, i.e., $ s\circ
(t\circ u) \cap (s\circ t) \circ u\neq \emptyset$,  for all
$s,t,u\in J$.

\item[(2)] $x\circ J=J=J\circ x$,  for all $x\in J$.
\end{itemize}

The dihedral group of order $2n$ is given by, $D_{2n}=\langle
a,b:$ $a^{n}=b^{2}=1,$ $ab=ba^{-1}\rangle$. We have constructed an
$H_v$-group $(D_{2n},\circ)$, where $D_{2n}$ is the dihedral group
and $\circ $ is the hyperoperation such that
$
\circ :D_{2n}\times D_{2n}\rightarrow \mathcal{P}^{\ast }(D_{2n})
$
defined by
\begin{equation}  \label{1}
x\circ y=\left\{ xy,xy^{-1},a,a^{-1},a^{2},a^{-2},b\right\} \text{\ \ for all }x,y\in D_{2n}\,,
\end{equation}
where on the right--hand side, $a$, $a^{-1}$, $a^{2},$ $a^{-2}$,
and $b$ are fixed elements of $D_{2n}$, while $x$, $y$ are any two
general elements of $D_{2n}$. In what follows, we discuss the
properties of commuting graph of this $H_v$-group. In the
remaining part of this paper, the $H_v$-group $(D_{2n},\circ)$ is
denoted by $H^{\natural }$. First of all, we have to find those
elements that commute with each other. The elements of $D_{2n}$
are of the type $a^{i}$, $a^{i}b$, for
$i\in\left\{1,2,\ldots,n\right\}$. Therefore, the compositions of
the elements of this $H_{v}$-group are possibly of the types
$a^{i}\circ a^{j}$, $a^{i}\circ a^{j}b$, $a^{i}b\circ a^{j}b$, for
$i,j\in\left\{1,2,\ldots,n\right\}$. We first consider the
compositions $a^{i}\circ a^{j}$, $a^{j}\circ a^{i}$ and find those
elements that commute with each other. Note that
\begin{align}
a^{i}\circ a^{j} &=\left\{ a^{i}\cdot a^{j},a^{i}\cdot a^{-j},a,a^{-1},a^{2},a^{-2},b\right\} \nonumber \\
&=\left\{ a^{i+j},a^{i-j},a,a^{-1},a^{2},a^{-2},b\right\},  \label{equ:2}
\end{align}
and
\begin{align}
a^{j}\circ a^{i} &=\left\{ a^{j}\cdot a^{i},a^{j}\cdot a^{-i},a,a^{-1},a^{2},a^{-2},b\right\} \nonumber \\
&=\left\{ a^{j+i},a^{j-i},a,a^{-1},a^{2},a^{-2},b\right\}. \label{equ:3}
\end{align}
If $j=i+1$, then the equations $(\ref{equ:2})$ and $(\ref{equ:3})$ become
\begin{eqnarray}
a^{i}\circ a^{j} & = & \left\{ a^{2i+1},a,a^{-1},a^{2},a^{-2},b\right\}, \label{equ:4} \\
a^{j}\circ a^{i} & = & \left\{ a^{2i+1},a,a^{-1},a^{2},a^{-2},b\right\}. \label{equ:5}
\end{eqnarray}
Thus $a^{i}$ commutes with $a^{i+1}$ for all $i\in\left\{1,2,\ldots,n\right\}$.
Similarly, for each $i\in\left\{1,2,\ldots,n\right\}$, we can see that $a^{i}$
commutes with $a^{j}$, where $j=i+1$, $i-1$, $i+2$, $i-2$, and also for
$j=\frac{n}{2}+i$, if $n$ is an even integer. In an analogous manner, one can check the
other compositions and find the elements that commute with each other.

Let $\Gamma $ be a subset of the $H_v$-group $(D_{2n},\circ)$. The
vertices of the commuting graph are the elements of $\Gamma $,
where any two different vertices $s,t\in \Gamma $ are joined by an
edge if $s\circ t=t\circ s$. The degree $deg_{G}(s)$ of a vertex
$s \in V(G)$ of a graph $G$ is the number of first neighbors of
$s$. The following two theorems explain about the degree of each
vertex in the commuting graph $G=C( H^{\natural },H^{\natural})$.

\begin{theorem}
\label{L1}Let $H^{\natural }= (D_{2n},\circ)$ be an $H_v$-group
for an even integer $n\geq 6$ and $G=C( H^{\natural
},H^{\natural})$ be a commuting graph. Then \begin{itemize}
\item[(1)] $\deg _{G}(a^{i}) =\left\{\begin{tabular}{ll}
$6$ &  \text{if}\ $i\neq n, n/2$ \\
$n+5$ & \text{if}\ $i=n, n/2$.
\end{tabular}
\right. $ \item[(2)] $ \deg _{G}(a^{i}b)
=\left\{\begin{tabular}{ll}
$8$ &  \text{if} \ $i\neq n, n/2$ \\
$7$ &  \text{if} \ $i=n, n/2$\,.
\end{tabular}
\right. $
\end{itemize}
\end{theorem}

\begin{proof}
(1) For an even integer $n\geq 6$, each $a^{i}$ commutes with
$a^{i+1}$, $a^{i-1}$, $a^{i+2}$, $a^{i-2}$, $a^{\frac{n}{2}+i}$.
Also $a^{i}$ commutes with $a^{i}b$ if $i\neq n,$ $\frac{n}{2}$
whereas $e,a^{\frac{n}{2}}$ commute with $a^{i}b$ for all
$i\in\left\{1,2,\ldots,n\right\}$.

(2) Each $a^{i}b$ commutes with $a^{i+1}b$, $a^{i-1}b$,
$a^{i+2}b$, $a^{i-2}b$, $a^{\frac{n}{2}+i}b$, $a^{i}$, $e$, and
$a^{\frac{n}{2}}$. Therefore, $\deg _{G}(a^{i}b) =8$ if  $i\neq
n,$ $\frac{n}{2}$ and $\deg _{G}(a^{i}b) =7$ if  $i= n,$
$\frac{n}{2}$.
\end{proof}

\begin{theorem}
\label{L2}Let $H^{\natural }= (D_{2n},\circ)$ be an $H_v$-group
for an odd integer $n\geq 5$ and $G=C(H^{\natural },H^{\natural})$
be a commuting graph. Then
\begin{itemize}
\item[(1)] $\deg _{G}( a^{i}) =\left\{\begin{tabular}{ll}
$5$ & \text{if} \ $i\neq n$ \\
$n+4$ & \text{if} \ $i=n$,
\end{tabular}
\right. $ \item[(2)] $\deg _{G}( a^{i}b)
=\left\{\begin{tabular}{ll}
$6$ & \text{if} \ $i\neq n$ \\
$5$ & \text{if} \ $i=n$\,.
\end{tabular}
\right. $
\end{itemize}
\end{theorem}

\begin{proof}
Relations (1) and (2) follow by straightforward calculations.
\end{proof}

Now, we present some algorithms to construct NSSD graphs from
these commuting graphs. These algorithms are written in the $GAP$
language.

\begin{algorithm}
\caption{Dihedral Group}
{ \bf Input} : n

{ \bf Output} : Dihedral group of order $2n$
\begin{enumerate}

\item  f := FreeGroup( "a", "b" );;
\item  g := $f / [ f.1^{n}, f.2^{2}, (f.1*f.2)^{2} ];$
\item  Unbind(a);
\item  a := g.1;; b := g.2;;  assign variables

\end{enumerate}
\end{algorithm}
Algorithm $(1)$ gives us dihedral group of order $2n$. Here in
this algorithm we have to give the input value of $n$ and we get
the dihedral group of order $2n$. Now, we give an algorithm to
define the hyperoperation given in Eq. $(\ref{1})$. This Algorithm $(2)$
gives us the product of two elements under the
hyperopeation defined in Eq. $(\ref{1})$.

\vspace{10mm}

\begin{algorithm}
\caption{Hyperoperation}
{ \bf Input} : two elements $x,y \in g$

{ \bf Output} : The image of (x,y) under the hyperoperation $"\circ"$, $i.e.,$ $x\circ y$.
\begin{enumerate}

\item  H := The function of $(x,y)$
\item  Define the local variable $"\circ"$
\item if x in g and y in g then
\item  $\circ$ := The hyperoperation defined as in Eq. $(\ref{1})$;
\item  fi; \text{ \ \ \ } return $\circ$; \text{ \ \ \ } end;

\end{enumerate}
\end{algorithm}

\vspace{10mm}

\begin{algorithm}
\caption{Adjacency Matrix}
{ \bf Input} : Any subset $U$ of this $H_{v}$-group

{ \bf Output} : The adjacency matrix for the commuting graph of $U$.
\begin{enumerate}

\item  T := function(U)
\item  local S, M, n, i, j, k;
\item  n is the order of the subset $U$;
\item  M is the identity matrix of order $n$;
\item  for i in [1..n-1] do
\item  for j in [i+1..n] do
\item  if the elements at $ith$ and $jth$ position in $U$ commutes then
\item  M[i][j] := 1;
\item  M[j][i] := 1; \text{ \ \ \ } fi;
\item  od; \text{ \ \ \ } od;
\item  for k in [1..n] do
\item  M[k][k] := 0;
\item  od; \text{ \ \ \ } return M; \text{ \ \ \ } end;

\end{enumerate}
\end{algorithm}

Algorithm $(3)$ presents the pseudo-code for the adjacency matrix
of a commuting graph $G=C(H^{\natural },U)$. Here subset $U$ is
the input value and the output value is the adjacency matrix for
the commuting graph of $U$. Now, the following algorithm $(4)$
shows that wether the commuting graph is NSSD graph or not.

\vspace{10mm}

\begin{algorithm}
\caption{NSSD Graph}
{ \bf Input} : An adjacency matrix

{ \bf Output} : Is corresponding graph NSSD graph or not.
\begin{enumerate}

\item  RemRowCol := function(M,c)
\item  local A, i; \text{ \ \ \ } A := StructuralCopy(M);
\item  for i in [1..Length(A)] do
\item  Remove $ith$ Row and $ith$ Column of matrix $M$;
\item  od; \text{ \ \ \ } return A; \text{ \ \ \ } end;
\item  IsNSSD := function(M) \text{ \ \ \ } local A, eig, eigsp, c, i, j;
\item  $"c"$ is the counter;
\item  $"eig"$ are the Eigenvalues of $M$;
\item  if 0 is an eigenvalue of $M$ then \text{ \ \ \ } return false;
\item     else c := c+1; \text{ \ \ \ } fi;
\item  for i in [1..Length(M)] do
\item  $"A"$ is the matrix obtained by deleting $ith$ Row and $ith$ Column of $M$;
\item  $"eigsp"$ are the Eigenvalues of $A$;
\item  if 0 is an eigenvalue of $A$ then \text{ \ \ \ } c := c+1; \text{ \ \ \ } fi;
\item  od; \text{ \ \ \ } if c=Length(M)+1 then
\item     return true; \text{ \ \ \ } else return false;
\item  fi; \text{ \ \ \ } end;

\end{enumerate}
\end{algorithm}
Algorithm $(4)$ is the pseudo-code for NSSD graph. In this
algorithm the input value is the adjacency matrix of a commuting
graph and it returns true if the corresponding graph is NSSD graph
otherwise it returns false. Using these algorithms present in this
paper, we can find the NSSD graphs. For example, if we consider
the dihedral group for an integer $n=4$ and define the
hyperoperation using algorithm $(2)$. Now, consider the subset $U
= \{ a, a^{3}, b, ab \}$ of the $H_{v}$-group $H^{\natural }=
(D_{8},\circ)$ and find the adjacency matrix for the commuting
graph $G=C( H^{\natural },U)$ using algorithm $(3)$, we get
\begin{center}
 $M=\begin{bmatrix}
    0 & 1 & 0 & 1 \\
    1 & 0 & 0 & 0 \\
    0 & 0 & 0 & 1 \\
    1 & 0 & 1 & 0 \\
 \end{bmatrix}$.
 \end{center}
When we use algorithm $(4)$ to check wether it is NSSD graph or not, it returns $"true"$. The corresponding graph is depicted in Fig. 1.

\begin{center}
\includegraphics[height=1.3cm,keepaspectratio]{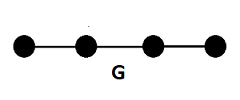}

\baselineskip=0.20in

\vspace{3mm}

{\bf Fig. 1.} NSSD graphs.
\end{center}

Now, consider the $H_{v}$-group $H^{\natural }= (D_{2n},\circ)$, for an integer $n \geq 2$. We present the following table that contains the number of NSSD graphs, obtained from the commuting graphs of the $H_{v}$-group $H^{\natural }= (D_{2n},\circ)$, with the help of above algorithms.

\vspace{5mm}

\begin{center}
\begin{tabular}{l|l|l|l}
 { \bf n } & { \bf Order of } & { \bf No. of subsets who's } & { \bf No. of } \\
 & { \bf graph } & { \bf commuting graph is NSSD } & { \bf NSSD graphs } \\ \hline
2 & \text{ \ \ \ \ \ }2 & \text{ \ \ \ \ \ \ \ \ \ \ \ \ \ \ \ }6 & \text{ \ \ \ \ \ \ \ }1 \\ \hline
3 & \text{ \ \ \ \ \ }2 & \text{ \ \ \ \ \ \ \ \ \ \ \ \ \ \ \ }11 & \text{ \ \ \ \ \ \ \ }1 \\
 & \text{ \ \ \ \ \ }4 & \text{ \ \ \ \ \ \ \ \ \ \ \ \ \ \ \ }2 & \text{ \ \ \ \ \ \ \ }1 \\ \hline
4 & \text{ \ \ \ \ \ }2 & \text{ \ \ \ \ \ \ \ \ \ \ \ \ \ \ \ }22 & \text{ \ \ \ \ \ \ \ }1 \\
 & \text{ \ \ \ \ \ }4 & \text{ \ \ \ \ \ \ \ \ \ \ \ \ \ \ \ }5 & \text{ \ \ \ \ \ \ \ }2 \\ \hline
5 & \text{ \ \ \ \ \ }2 & \text{ \ \ \ \ \ \ \ \ \ \ \ \ \ \ \ }29 & \text{ \ \ \ \ \ \ \ }1 \\
 & \text{ \ \ \ \ \ }4 & \text{ \ \ \ \ \ \ \ \ \ \ \ \ \ \ \ }54 & \text{ \ \ \ \ \ \ \ }2 \\ \hline
6 & \text{ \ \ \ \ \ }2 & \text{ \ \ \ \ \ \ \ \ \ \ \ \ \ \ \ }46 & \text{ \ \ \ \ \ \ \ }1 \\
 & \text{ \ \ \ \ \ }4 & \text{ \ \ \ \ \ \ \ \ \ \ \ \ \ \ \ }84 & \text{ \ \ \ \ \ \ \ }2 \\ \hline
7 & \text{ \ \ \ \ \ }2 & \text{ \ \ \ \ \ \ \ \ \ \ \ \ \ \ \ }41 & \text{ \ \ \ \ \ \ \ }1 \\
 & \text{ \ \ \ \ \ }4 & \text{ \ \ \ \ \ \ \ \ \ \ \ \ \ \ \ }262 & \text{ \ \ \ \ \ \ \ }2 \\
 & \text{ \ \ \ \ \ }6 & \text{ \ \ \ \ \ \ \ \ \ \ \ \ \ \ \ }374 & \text{ \ \ \ \ \ \ \ }7 \\
 & \text{ \ \ \ \ \ }8 & \text{ \ \ \ \ \ \ \ \ \ \ \ \ \ \ \ }130 & \text{ \ \ \ \ \ \ \ }15 \\
 & \text{ \ \ \ \ \ }10 & \text{ \ \ \ \ \ \ \ \ \ \ \ \ \ \ \ }4 & \text{ \ \ \ \ \ \ \ }1 \\ \hline
8 & \text{ \ \ \ \ \ }2 & \text{ \ \ \ \ \ \ \ \ \ \ \ \ \ \ \ }62 & \text{ \ \ \ \ \ \ \ }1 \\
 & \text{ \ \ \ \ \ }4 & \text{ \ \ \ \ \ \ \ \ \ \ \ \ \ \ \ }409 & \text{ \ \ \ \ \ \ \ }2 \\
 & \text{ \ \ \ \ \ }6 & \text{ \ \ \ \ \ \ \ \ \ \ \ \ \ \ \ }416 & \text{ \ \ \ \ \ \ \ }7 \\
 & \text{ \ \ \ \ \ }8 & \text{ \ \ \ \ \ \ \ \ \ \ \ \ \ \ \ }80 & \text{ \ \ \ \ \ \ \ }11 \\
 & \text{ \ \ \ \ \ }10 & \text{ \ \ \ \ \ \ \ \ \ \ \ \ \ \ \ }4 & \text{ \ \ \ \ \ \ \ }1
\end{tabular}
\captionof{table}{Number of NSSD graphs for different values of
$n$.}\label{tab1}
\end{center}
By using Table \ref{tab1}, we calculate the number of NSSD graphs
of different orders for different values of $n$. In addition, we
compute the number of subsets, who's commuting graph is NSSD.
Similarly, from these commuting graphs one can find more NSSD
graphs of higher order by choosing greater value of $n$. Now, we
present some NSSD molecular graphs obtained from these commuting
graphs.


\section{NSSD Molecular Graphs}

As mentioned in previous section, NSSD graphs are encountered
within a theory of conductivity of organic substances
\cite{new1,14}. In view of this, it is of particular interest to
design NSSD graph that are molecular graphs, i.e., graphs whose
vertices and edges pertain to carbon atoms and carbon--carbon
bonds, respectively \cite{17,18a,new3}. Hyperstructure theory has been
earlier much used in the chemistry, see \cite{21,21a,21b,21c}. In
this section our main purpose is to construct NSSD molecular
graphs from the above described commuting graphs. The following
theorems related to construct NSSD graphs from the commuting
graphs of a non-abelian group $\Omega$.

\begin{theorem}
Let $G_{1} = C(\Omega, U)$ and $G_{2} = C(\Omega, V)$ be two commuting graphs, such that $G_{2}$ is an empty graph and $|G_{1}| = |G_{2}|$. If each element of $V$ commutes with exactly one element of $U$, then the commuting graph $G^{\prime} = C(\Omega, U \cup V)$ is an NSSD graph.
\end{theorem}

\begin{proof}
Since each element of the subset $V$ commutes with exactly one
element of $U$ and $G_{2}$ is an empty graph, it follows that each
vertex in $V$ is a pendent vertex of the commuting graph
$G^{\prime}$. If $v$ is a pendent vertex of a graph $G^{\prime}$,
adjacent to the vertex $u$, then \cite{6a, 6b}
\begin{equation}
\eta(G^{\prime}) = \eta(G^{\prime} -v-u). \label{eq 1}
\end{equation}
If we apply Eq. $(\ref{eq 1})$ to each but one pendent vertex of
the graph $G^{\prime}$, then we get a connected graph with two
vertices. Therefore, nullity of the graph $G^{\prime}$ is zero. So
$G^{\prime}$ is a non-singular graph.

Now, consider the vertex deleted subgraph $G^{\prime}-x$. If $x
\in V$, then $x$ is a pendent vertex, so applying the Eq.
$(\ref{eq 1})$ to each pendent vertex of the graph $G^{\prime}-x$,
we obtain a graph with single vertex. Therefore, the nullity of
$G^{\prime}-x$ is $1$. If $x \in U$, then there exists an isolated
vertex of the graph $G^{\prime}-x$, Therefore, the nullity of
$G^{\prime}-x$ is $1$. Hence $G^{\prime}$ is a non-singular graph
with a singular deck.
\end{proof}

\begin{theorem}
If the commuting graphs $G_{1}= C(\Omega, U)$ and $G_{2}= C(\Omega, V)$ are two NSSD graphs, such that there exists exactly one element $u \in U$ that commutes with exactly one element $v \in V$, then the commuting graph $G^{\prime}= C(\Omega, T)$ is an NSSD graph, where $T = U \cup V$.
\end{theorem}

\begin{proof}
Clearly,  $G^{\prime}$ is obtained by joining a vertex $u \in
G_{1}$ with a vertex $v \in G_{2}$. The following relation gives
the characteristic polynomial \cite{new3,18a}
\begin{equation}
P(G^{\prime}, \lambda) = P(G_{1}, \lambda) P(G_{2}, \lambda) - P(G_{1}-u, \lambda)P(G_{2}-v, \lambda). \label{eq 2}
\end{equation}
Both graphs $G_{1}$ and $G_{2}$ are NSSD graphs, so they are
non-singular, i.e., $P(G_{1}, 0) \neq 0$, $P(G_{2},0) \neq 0$.
Moreover, each vertex deleted subgraph is singular, so $P(G_{1}-u,
0) = 0$ and $P(G_{2}-v, 0) = 0$. Consequently, we get
$P(G^{\prime}, 0) \neq 0$, and this implies that $G^{\prime}$ is
non-singular. Now, consider the vertex deleted subgraph
$G^{\prime} - x$. If $x = u$, then $G^{\prime} - u$ is singular,
because $G_{1}-u$ is singular. Similarly, if $x = v$, then the
subgraph $G^{\prime} - v$ is singular. Let $x \in T$, such that $x
\neq u, v$. Assume that $x \in U$, then from Eq. $(\ref{eq 2})$
\begin{equation*}
P(G^{\prime}-x, \lambda) = P(G_{1}-x, \lambda) P(G_{2}, \lambda) - P(G_{1}-x-u), \lambda) P(G_{2}-v, \lambda).
\end{equation*}
We have $P(G_{1}-x, 0) = 0$, because $G_{1}$ is an NSSD graph.
Therefore, $P(G^{\prime}-x, 0) = 0$, which shows that the subgraph
$G^{\prime}-x$ is singular. Hence $G^{\prime}$ is an NSSD graph.
\end{proof}

Now, using these results and the algorithms presented in section
$(2)$, we construct the NSSD molecular graphs from the commuting
graphs of the $H_{v}$-group $H^{\natural}$. Consider the
$H_v$-group $H^{\natural}=(D_{2n},\circ)$, where $D_{2n}$ is the
dihedral group for $n=16$ and $\circ$ is the hyperoperation
defined as in Eq. $(\ref{1})$. In addition, define the following
sets of vertices for which the commuting graphs give NSSD
molecular graphs with 2, 4, and 6 vertices:
\begin{eqnarray*}
\Gamma _{1} &=&\left\{a,a^{3}\right\}, \\
\Gamma _{2} &=&\left\{a^{5},a^{7},a^{8},a^{10}\right\}, \\
\Gamma _{3} &=&\left\{a,a^{3},a^{5},a^{7},a^{8},a^{10}\right\}, \\
\Gamma _{4} &=&\left\{a,a^{3},a^{5},a^{6},a^{5}b,a^{7}b\right\}, \\
\Gamma _{5} &=&\left\{a^{2},a^{3},a^{5},a^{7},a^{9},a^{10}\right\}, \\
\Gamma _{6} &=&\left\{a,a^{2},a^{10},ab,a^{2}b,a^{4}b\right\}, \\
\Gamma _{7} &=&\left\{a,a^{2},a^{3},a^{5},a^{15},a^{2}b\right\}.
\end{eqnarray*}
Here the commuting graphs $G_i=C(H^{\natural},\Gamma _i) \ , \
i=1,2,\ldots,7$, result in the following NSSD graphs.

\vspace{3mm}

\begin{center}
\includegraphics[height=4cm,keepaspectratio]{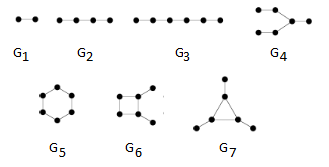}

\baselineskip=0.20in

\vspace{4mm}

{\bf Fig. 2.} NSSD graphs with 2, 4, and 6 vertices.
\end{center}

\vspace{4mm}

\baselineskip=0.30in For instance, we specify the construction of
the graph $G_3$, whose vertex set is $\Gamma_{3}$. Since the
graphs $G_{1}$, $G_{2}$ are NSSD and only one element $a^{3} \in
\Gamma_{1}$ commutes with exactly one element $a^{5} \in
\Gamma_{2}$. So, the commuting graph corresponding to $\Gamma
_{3}=\Gamma _{1}\cup \Gamma _{2}$ is NSSD. Thus, the graph $G_3$
is the path of the form $a-a^{3}-a^{5}-a^{7}-a^{8}-a^{10}$. In an
analogous manner, one can establish the structure of the remaining
graphs from the set $\{\Gamma_i \ | \ i=1,2,\ldots,7\}$. One can
determine these graphs using algorithm $(4)$. We now list the sets
of vertices for which the commuting graphs yield NSSD graphs with
8 vertices.
\begin{eqnarray*}
\Gamma _{8} &=&\left\{a,a^{2},a^{3},a^{5},a^{15},a^{2}b,a^{5}b,a^{7}b\right\}, \\
\Gamma _{9} &=&\left\{a,a^{3},a^{5},a^{7},a^{9},a^{10},a^{3}b,a^{5}b\right\}, \\
\Gamma _{10} &=&\left\{a^{5},a^{6},a^{5}b,a^{6}b,a^{13}b,a^{14}b,a^{12}b,a^{12}\right\}, \\
\Gamma _{11} &=&\left\{a^{2},a^{3},a^{5},a^{7},a^{9},a^{10},a^{2}b,a^{7}b\right\}, \\
\Gamma _{12}
&=&\left\{a,a^{3},a^{5},a^{7},a^{9},a^{10},ab,a^{7}b\right\}.
\end{eqnarray*}
For these sets of vertices the commuting graphs $G_{i}=C(H^{\natural},\Gamma _i) \ , \
i=8,9,\ldots,12$, are the following NSSD graphs.

\vspace{3mm}

\begin{center}
\includegraphics[height=5cm,keepaspectratio]{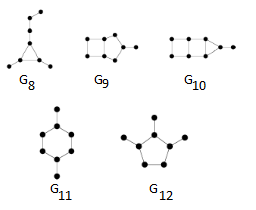}

\baselineskip=0.20in

\vspace{6mm}

{\bf Fig. 3.} NSSD graphs with 8 vertices.
\end{center}

\vspace{3mm}

\baselineskip=0.30in

For instance, $\Gamma _{8} =\left\{a,a^{2},a^{3},a^{5},a^{15},a^{2}b,a^{5}b,a^{7}b\right\}$
is the set of vertices for the commuting graph $G_{8}$. Since the commuting graphs of the subsets $\Gamma _{7}$ and $\Gamma^{\prime} = \{a^{5}b, a^{7}b\}$ NSSD. Also there exists only one element $a^{5}b \in \Gamma^{\prime}$ that commutes with only one element $a^{2}b \in \Gamma _{7}$. Thus the commuting graph of the subset $\Gamma _{8}=\Gamma _{7} \cup \Gamma^{\prime}$ is NSSD. These structural features fully determine the NSSD graph $G_{8}$. The other graphs can be analysed and determined in a similar manner.

The following sets of vertices pertain to commuting graphs resulting in NSSD
molecular graphs with 10 vertices.
\begin{eqnarray*}
\Gamma _{13} &=&\left\{a,a^{3},a^{5},a^{7},a^{9},a^{10},ab,a^{5}b,a^{7}b,a^{13}b\right\}, \\
\Gamma _{14} &=&\left\{a,a^{3},a^{5},a^{7},a^{9},a^{10},a^{3}b,a^{5}b,a^{11}b,a^{13}b\right\}, \\
\Gamma _{15} &=&\left\{a^{2},a^{3},a^{5},a^{6},a^{7},a^{9},a^{10},a^{2}b,a^{3}b,a^{14}\right\}, \\
\Gamma _{16} &=&\left\{a^{5},a^{6},a^{7},a^{9},a^{12},a^{5}b,a^{6}b,a^{13}b,a^{14}b,a^{12}b\right\}, \\
\Gamma _{17} &=&\left\{a^{2},a^{3},a^{5},a^{7},a^{9},a^{10},a^{12},a^{13},a^{15},a^{12}b\right\}, \\
\Gamma _{18}
&=&\left\{a,a^{3},a^{4},a^{6},a^{7},a^{9},a^{11},a^{13},a^{14},a^{14}b\right\}.
\end{eqnarray*}
Consider now the $H_{v}$-group $H^{\natural}= (D_{2n},\circ)$,
where $D_{2n}$ is the dihedral group for $n=20$ and $\circ$ is the
hyperoperation, and define the following set of vertices of this
$H_{v}$-group for which the commuting graphs give NSSD graphs with
10 vertices.
\begin{eqnarray*}
\Gamma _{19} &=&\left\{a^{3},a^{5},a^{7},a^{9},a^{11},a^{12},a^{14},a^{15},a^{9}b,a^{11}b\right\}, \\
\Gamma _{20} &=&\left\{a,a^{2},a^{3},a^{5},a^{6},a^{7},a^{9},a^{12},ab,a^{6}b\right\}, \\
\Gamma _{21} &=&\left\{a,a^{2},a^{3},a^{5},a^{6},a^{8},a^{19},a^{2}b,a^{5}b,a^{6}b\right\}, \\
\Gamma _{22} &=&\left\{a,a^{2},a^{4},a^{5},a^{11},a^{2}b,a^{3}b,a^{5}b,a^{11}b,a^{12}b\right\}, \\
\Gamma _{23} &=&\left\{a^{3},a^{5},a^{7},a^{9},a^{11},a^{12},a^{14},a^{15},a^{19},a^{7}b\right\}, \\
\Gamma _{24}
&=&\left\{a^{3},a^{5},a^{7},a^{9},a^{11},a^{12},a^{14},a^{15},a^{7}b,,a^{12}b\right\}.
\end{eqnarray*}
The commuting graphs $G_i=C(H^{\natural},\Gamma_i) , \
i=13,14,\ldots,24$, lead to the NSSD graphs depicted in Figs. 4
and 5.

\vspace{10mm}

\begin{center}
\includegraphics[height=5cm,keepaspectratio]{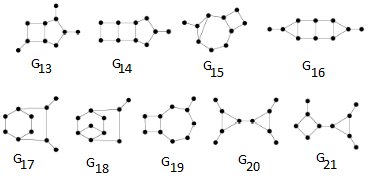}

\baselineskip=0.20in

\vspace{6mm}

{\bf Fig. 4.} NSSD graphs with 10 vertices.
\end{center}

\begin{center}
\includegraphics[height=4cm,keepaspectratio]{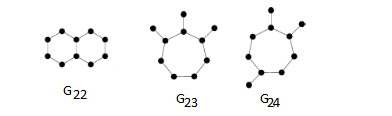}

\baselineskip=0.20in

\vspace{6mm}

{\bf Fig. 5.} NSSD graphs with 10 vertices.
\end{center}

\vspace{5mm}

\baselineskip=0.30in
In order to construct NSSD graphs with 12
vertices from these commuting graphs, consider the $H_{v}$-group
$H^{\natural}= (D_{2n},\circ)$, for $n=20$ and define the
following sets of vertices.

\vspace{3mm}

\begin{eqnarray*}
\Gamma _{25} &=&\left\{a,a^{2},a^{3},a^{5},a^{19},a^{2}b,a^{5}b,a^{7}b,a^{9}b,a^{11}b,a^{13}b,a^{14}b\right\}, \\
\Gamma _{26} &=&\left\{a,a^{2},a^{3},a^{5},a^{6},a^{8},a^{16},a^{19},a^{2}b,a^{8}b,a^{9}b,a^{11}b\right\}, \\
\Gamma _{27} &=&\left\{a,a^{2},a^{3},a^{5},a^{6},a^{8},a^{16},a^{19},a^{2}b,a^{8}b,a^{9}b,a^{16}b\right\}, \\
\Gamma _{28} &=&\left\{a,a^{2},a^{3},a^{5},a^{6},a^{8},a^{19},a^{2}b,a^{8}b,a^{9}b,a^{11}b,a^{18}b\right\}, \\
\Gamma _{29} &=&\left\{a,a^{2},a^{3},a^{5},a^{6},a^{19},a^{2}b,a^{6}b,a^{8}b,a^{9}b,a^{16}b,a^{17}b\right\}, \\
\Gamma _{30} &=&\left\{a,a^{2},a^{3},a^{5},a^{6},a^{8},a^{19},a^{2}b,a^{6}b,a^{7}b,a^{9}b,a^{16}b\right\}, \\
\Gamma _{31} &=&\left\{a,a^{2},a^{3},a^{5},a^{6},a^{15},a^{19},a^{2}b,a^{6}b,a^{8}b,a^{9}b,a^{15}b\right\}, \\
\Gamma _{32} &=&\left\{a,a^{2},a^{3},a^{5},a^{7},a^{19},a^{2}b,a^{4}b,a^{7}b,a^{9}b,a^{11}b,a^{14}b\right\}, \\
\Gamma _{33} &=&\left\{a,a^{2},a^{3},a^{5},a^{6},a^{8},a^{15},a^{19},a^{2}b,a^{6}b,a^{7}b,a^{15}b\right\}, \\
\Gamma _{34} &=&\left\{a,a^{2},a^{3},a^{5},a^{7},a^{15},a^{16},a^{19},a^{2}b,a^{4}b,a^{7}b,a^{14}b\right\}, \\
\Gamma _{35} &=&\left\{a,a^{2},a^{3},a^{5},a^{7},a^{18},a^{19},a^{2}b,a^{4}b,a^{7}b,a^{14}b,a^{18}b\right\} \\
\Gamma _{36} &=&\left\{a,a^{2},a^{3},a^{5},a^{6},a^{11},a^{17},a^{2}b,a^{6}b,a^{7}b,a^{8}b,a^{10}b\right\}, \\
\Gamma _{37} &=&\left\{a,a^{2},a^{3},a^{5},a^{7},a^{8},a^{11},a^{17},a^{2}b,a^{8}b,a^{17}b,a^{18}b\right\}, \\
\Gamma _{38} &=&\left\{a,a^{2},a^{3},a^{5},a^{7},a^{8},a^{19},a^{2}b,a^{7}b,a^{8}b,a^{10}b,a^{17}b\right\}, \\
\Gamma _{39}
&=&\left\{a^{2},a^{5},a^{6},a^{7},a^{9},a^{11},a^{12},a^{15},a^{2}b,a^{6}b,a^{12}b,a^{14}b\right\}.
\end{eqnarray*}

\vspace{3mm}
The commuting graphs $G_i=C(H^{\natural},\Gamma_i) \
, \ i=25,26,\ldots,39$, yield the NSSD graphs with 12 vertices
depicted in Figs. 6 and 7.

\vspace{10mm}

\begin{center}
\includegraphics[height=4cm,keepaspectratio]{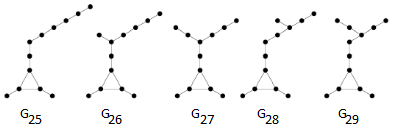}

\baselineskip=0.20in

\vspace{6mm}

{\bf Fig. 6.} NSSD graphs with 12 vertices.
\end{center}

\vspace{10mm}

\begin{center}
\includegraphics[height=8cm,keepaspectratio]{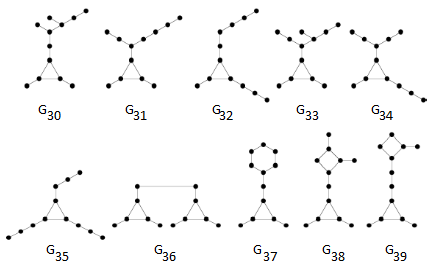}

\baselineskip=0.20in

\vspace{6mm}

{\bf Fig. 7.} NSSD graphs with 12 vertices.
\end{center}

\vspace{10mm}

\baselineskip=0.30in
At the end, consider the $H_{v}$-group
$H^{\natural}= (D_{2n},\circ)$, for $n=20$, and define the
following sets of vertices to construct the NSSD graphs with 12
and 14 vertices.
\begin{eqnarray*}
\Gamma _{40} &=&\left\{a,a^{2},a^{3},a^{5},a^{8},a^{11},a^{2}b,a^{7}b,a^{8}b,a^{9}b,a^{11}b,a^{17}b\right\}, \\
\Gamma _{41} &=&\left\{a^{3},a^{5},a^{7},a^{9},a^{11},a^{12},a^{14},a^{15},a^{19},a^{7}b,a^{11}b,a^{15}b\right\}, \\
\Gamma _{42}
&=&\left\{a^{2},a^{3},a^{7},a^{8},a^{12},a^{13},a^{7}b,a^{8}b,a^{12}b,a^{13}b,a^{14}b,a^{16}b,a^{17}b,a^{18}b\right\}.
\end{eqnarray*}
Also, consider the $H_{v}$-group $H^{\natural}= (D_{2n},\circ)$,
for $n=24$ and define the following set of vertices to construct
an NSSD graph with 16 vertices:
$$
\Gamma_{43} = \left\{a^{3},a^{4},a^{9},a^{10},a^{15},a^{16},a^{17},a^{3}b,a^{4}b,
a^{9}b,a^{10}b,a^{17}b,a^{18}b,a^{20}b,a^{21}b,a^{22}b\right\}
$$
The commuting graphs $G_i=C(H^{\natural},\Gamma_i) , \
i=40,\ldots,43$, are the following NSSD molecular graphs with 12,
14, and 16 vertices:

\vspace{10mm}

\begin{center}
\includegraphics[height=4.5cm,keepaspectratio]{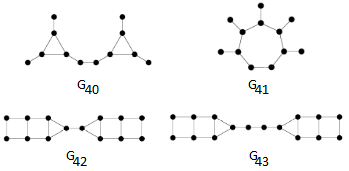}

\baselineskip=0.20in

\vspace{3mm}

{\bf Fig. 8.} NSSD graphs with 12, 14, and 16 vertices.
\end{center}

\vspace{5mm}

\baselineskip=0.25in

\begin{remark}
There are a lot of NSSD graphs but we have shown only a few here. One can find NSSD graphs of higher order by choosing high values of $n$.
\end{remark}

\section{Conclusion}
In this article we have defined an $H_{v}$-group and discussed its commuting graph. We have constructed NSSD molecular graphs from the commuting graph of this $H_{v}$-group. Also we have defined an algorithm that can construct NSSD graphs. In this paper, we have considered a hyperoperation on dihedral group given in Eq. $(\ref{1})$. For feature work in this direction on can use another hyperoperation and determine different NSSD graphs.

\section*{Acknowledgments} We  thank the referees for useful suggestions. This research is partially funded through Quaid-i-Azam University grant URF-2015


\begin{thebibliography}{99}
\bibitem{17} Graovac A., Gutman I., Trinajsti\'c N., Topological Approach to the
        Theory of Conjugated Molecules, {\it Springer, Berlin,\/} 1977.

\bibitem{18} Gutman I., Trinajsti\'c N., Graph theory and molecular orbitals,
        {\it Topics Curr. Chem.,\/} 1973, 42, 49--93.

\bibitem{new1} Fowler P.W., Pickup B.T., Todorova T.Z.,
        Borg M., Sciriha I., Omni-conducting and omni-insulating
        molecules, {\it J. Chem. Phys.,\/} 2014, 140, 054--115.

\bibitem{new2} Cvetkovi\'c D., Rowlinson P., Simi\'c S., An Introduction to the Theory of Graph Spectra,
        {\it Cambridge Univ. Press, Cambridge,\/} 2010.

\bibitem{15} Farrugia A., Gauci J.B., Sciriha I.,
        Non-singular graphs with a singular deck,
        {\it Discrete Appl. Math.,\/} 2016, 202, 50--57.

\bibitem{14} Gutman I., Furtula B., Farrugia A., Sciriha I., Constructing NSSD
        molecular graphs, {\it Croat. Chem. Acta,\/} 2016, 89, 449--454.

\bibitem{new4} Sciriha I., Graphs with a common eigenvalue deck,
        {\it Linear Algebra Appl.,\/} 2009, 430, 78--85.

\bibitem{new5} Farrugia A., Gauci J.B., Sciriha I., On the inverse of the adjacency matrix of a graph,
        {\it Special Matrices,\/} 2013, 1, 28--41.

\bibitem{10} Marty F., Sur une generalization de la notio de groupe,
        {\it in: Proceeding, 8th Congress Mathematics Scandenaves Stockholm,\/} 1934, 45--49.

\bibitem{9} Nieminen J., Join space graphs, {\it J. Geom.,\/} 1988, 33, 99--103.

\bibitem{6} Corsini P., Leoreanu--Fotea V., Applications of Hyperstructure Theory,
         {\it Kluwer, Dordrecht,\/} 2003.

\bibitem{8} Davvaz B., Remarks on weak hypermodules, {\it Bull. Korean Math. Soc.,\/}
         1999, 36, 599--608.

\bibitem{13} Vougiouklis T., The fundamental relation in hyperrings. The general hyperfield, in: Algebraic
         Hyperstructures and Applications, {\it World Scientific, Teaneck\/}, 1991, 203--211.

\bibitem{12} Vougiouklis T., Hyperstructures and Their Representations, {\it Hadronic Press, Palm Harbor,\/}
         1994.

\bibitem{new6} Farrugia A., Edge construction of molecular NSSDs,
        {\it Discrete Applied Mathematics,\/} 2019, 266, 130--140.

\bibitem{new7} Farrugia A., Sciriha I., Triangles in inverse NSSD graphs,
        {\it Linear and Multilinear Algebra,\/} 2018, 66, no. 3, 540--546.

\bibitem{1} Abdollahi A., Akbari S., Maimani H.R., Non-commuting graph of a group,
        {\it J. Algebra,\/} 2006, 298, 468--492.

\bibitem{2} Anderson D.F., Badawi A., The total graph of a commutative ring,
        {\it J. Algebra,\/} 2008, 320, 2706--2719.

\bibitem{3} Anderson D.F., Livingston P., The zero-divisor graph of a commutative
        ring, {\it J. Algebra,\/} 1999, 217, 434--447.

\bibitem{4} Bundy D., The connectivity of commuting graphs, {\it J. Comb. Theory A,\/}
         2006, 113, 995--1007.

\bibitem{5} Chelvam T.T., Selvakumar K., Raja S., Commuting graphs on dihedral group,
        {\it Turk. J. Math. Comput. Sci.,\/} 2011, 2, 402--406.

\bibitem{11} Hayat U., Ali F., Nolla de Celis A., Commuting graphs on finite subgroups of SL$(2,\mathbb{C})$ and Dynkin diagrams, arXiv:1703.02480v1 [math.GR], 2017.

\bibitem{18a} Gutman I., Polansky O.E., Cyclic conjugation and the Huckel molecular orbital model,
        {\it Theor. Chim. Acta,\/} 1981, 60, 203--226.

\bibitem{new3} Gutman I., Polansky O.E., Mathematical Concepts in Organic Chemistry,
        {\it Springer, Berlin,\/} 1986.

\bibitem{21a} Davvaz B., Nezad A.D., Benvidi A., Chain reactions as
        experimental examples of ternary algebraic hyperstructures,
        {\it MATCH Commun. Math. Comput. Chem.,\/} 2011, 65, 491--499.

\bibitem{21b} Davvaz B., Nezhad A.D., Dismutation reactions as experimental
        verifications of ternary algebraic hyperstructures, {\it MATCH Commun.
        Math. Comput. Chem.,\/} 2012, 68, 551--559.

\bibitem{21} Davvaz B., Nezhad A.D., Benvidi A., Chemical hyperalgebra: Dismutation
        reactions, {\it MATCH Commun. Math. Comput. Chem.,\/} 2012, 67, 55--63.

\bibitem{21c} Davvaz B., Nezhad A.D., Mazloum--Ardakani M., Chemical hyperalgebra:
        Redox reactions, {\it MATCH Commun. Math. Comput. Chem.,\/} 2014, 71, 323--331.

\bibitem{6a} Cvetkovi\'c D., Gutman I., Trinajsti\'c N., Graph theory and molecular orbitals. $II^{*}$,
        {\it Croat. Chem. Acta,\/} 1972, 44, 365--374.

\bibitem{6b} Cvetkovi\'c D., Gutman I., Trinajsti\'c N., Graphical studies on the relations between the
        structure and reactivity of conjugated systems: the role of non-bonding molecular orbitals,
        {\it J. Mol. Struct.,\/} 1975, 28, 289--303.



\end{thebibliography}
\end{document}